\documentclass[a4paper]{elsarticle}
\usepackage[english]{babel}
\usepackage{amsfonts}
\usepackage{bm}
\usepackage{xcolor}
\usepackage{amsthm}
\usepackage{latexsym}
\usepackage{amsmath}
\usepackage{amssymb}
\usepackage{array}
\usepackage{arydshln}
\usepackage{tikz}
\usepackage{graphicx}
\usetikzlibrary{matrix,automata,arrows,decorations.pathmorphing,backgrounds,positioning,fit,petri}

\makeatletter
\def\ps@pprintTitle{%
 \let\@oddhead\@empty
 \let\@evenhead\@empty
 \def\@oddfoot{}%
 \let\@evenfoot\@oddfoot}
\makeatother

\definecolor{te}{HTML}{C11B17}
\newcommand{\ww}[1]{\textcolor{black}{#1}}
\newcommand{\cc}[1]{\textcolor{black}{#1}}
\newcommand{\lu}[1]{\textcolor{black}{#1}}

\newtheorem{theorem}{Theorem}
\newtheorem*{result}{Theorem A}
\newtheorem*{result2}{Theorem B}

\newtheorem{corollary}{Corollary}

\newtheorem{definition}{Definition}

\newtheorem{lemma}{Lemma}

\newtheorem{proposition}{Proposition}
\newtheorem{remark}{Remark}

\numberwithin{equation}{section}

\begin{document}
\title{On the abelian complexity of the Rudin-Shapiro sequence\tnoteref{t1}}
\tnotetext[t1]{This work was supported by the Fundamental Research Funds for the Central Universities (Project No. 2662015QD016) and NSFC (Nos. 11401188, 11371156, 11431007). The research of Wen Wu was also partially supported by the Academy of Finland, the Centre of Excellence in Analysis and Dynamic Research.}
\author[hust]{Xiaotao L\"{u}}
\ead{M201270021@hust.edu.cn}
\author[hzau]{Jin Chen}
\ead{wind.golden@gmail.com}
\author[hust]{Zhixiong Wen}
\ead{zhi-xiong.wen@mail.hust.edu.cn}
\author[hubu,oulu]{Wen Wu\corref{cor1}}
\ead{hust.wuwen@gmail.com}
\cortext[cor1]{Corresponding author.}
\address[hust]{School of Mathematics and Statistics, Huazhong University of Science and Technology, Wuhan 430074, China}
\address[hzau]{College of Science, Huazhong Agricultural University, Wuhan 430070, China}
\address[hubu]{Faculty of Mathematics and Statistics, Hubei University, Wuhan 430062, China}
\address[oulu]{Mathematics, P.O. Box 3000, 90014 University of Oulu, Finland}
\date{}                                  

\begin{abstract}
In this paper, we study the abelian complexity of the Rudin-Shapiro sequence and a related sequence. We show that these two sequences
share the same complexity function $\rho(n)$ which satisfies certain recurrence relations. As a consequence, the abelian complexity function is
$2$-regular. Further, we prove that the box dimension of the graph of the asymptotic function $\lambda(x)$ is $3/2$ where $\lambda(x)=\lim_{k\to\infty}\rho(4^{k}x)/\sqrt{4^{k}x}$ and $\rho(x)=\rho(\lfloor x\rfloor)$ for any $x> 0$.
\end{abstract}
\maketitle


\section{Introduction}
The abelian complexity of infinite words has been examined by Coven and Hedlund in
\cite{CH73} as an alternative way to characterize periodic sequences and Sturmian sequences.
Richomme, Saari, and Zamboni introduced this notion formally in \cite{RSZ11} which
initiated a general study of the abelian complexity of infinite words over finite alphabets.
For example, the abelian complexity functions of some notable sequences, such as the Thue-Morse sequence and all Sturmian sequences, were studied in \cite{RSZ11} and \cite{CH73} respectively. There also many other works devoted to this subject, see \cite{BBT11,BR13,CR11, RSZ10} and references therein. In the following, we will give the definition of the abelian complexity.

Let $\mathbf{w}=w(0)w(1)w(2)\cdots $ be an infinite sequence on a finite alphabet $\mathcal{A}$.  Denote by ${\mathcal{F}_{\mathbf{w}}(n)}$ the set of all
factors of $\mathbf{w}$ of length $n$, i.e., \[{\mathcal{F}_{\mathbf{w}}}(n): = \{w(i)w(i + 1)\cdots w(i + n - 1) : i \geq 0 \}.\]
Two finite words $u$, $v$ over a same alphabet $\mathcal{A}$ is \emph{abelian equivalent} if $|u|_{a}=|v|_{a}$ for any letter $a\in\mathcal{A}$. The abelian equivalent induces an equivalent relation, denoted by  $\sim_{ab}$. Now we are ready to state the definition of the abelian complexity.
\begin{definition}
The \emph{abelian complexity function} ${\rho_{\mathbf{w}}}:~\mathbb{N} \to \mathbb{N}$ of $\mathbf{w}$ is defined by
\[\rho_{\mathbf{w}}(n) := \# (\mathcal{F}_{\mathbf{w}}(n)/\sim_{ab}).\]
\end{definition}

First part of this paper is devoted to study the regularity of the abelian complexity of the Rudin-Sharpiro sequence
 $\mathbf{r}=r(0)r(1)r(2)\cdots$  whose generating function $R(z):=\sum_{n\geq 0}r(n)z^{n}$ satisfies the Mahler type functional equation
\[R(z)+R(-z)=2R(z^{2}).\]
Denote the coefficient sequence of $R(-z)$ by $\mathbf{r^{\prime}}$. To state our result, we shall recall the definition of $k$-regular and automatic sequences. (For more detail, see \cite{ALL}.)
\begin{definition}
Let $k\geq 2$ be an integer. The \emph{$k$-kernel} of an infinite sequence $\mathbf{w}=(w(n))_{n\geq 0}$ is the set of sub-sequences
$${\mathbf{K}_k}(\mathbf{w}):=\{ (w({k^e}n + c))_{n \ge 0}~|~ {e \ge 0,0 \le c < k^e}\} .$$
$\mathbf{w}$ is \emph{$k$-automatic} if $\mathbf{K}_k(\mathbf{w})$ is finite. If the $\mathbb{Z}$-module generated by its $k$-kernel is
finitely generated, then $\mathbf{w}=(w(n))_{n\geq 0}$ is \emph{$k$-regular}.
\end{definition}

Now we state our first result.

\begin{result}
The abelian complexity of the Rudin-Shapiro sequence $\mathbf{r}$, which is the same as the abelian complexity of $\mathbf{r}^{\prime}$, is $2$-regular.
\end{result}

\begin{figure}[htbp]\label{fig:lambda}
\centering
\includegraphics[scale=.4]{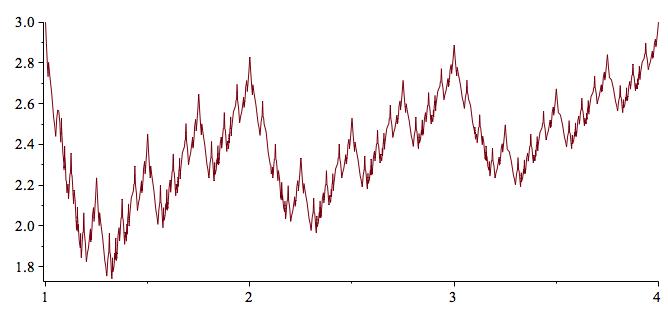}
\caption{The graph of $\lambda(x)$ for $x\in [1,4]$.}
\end{figure}

In the second part, in sprite by the work of Brillhart, Erd\H{o}s and Morton  \cite{BEM}, we  study the limit function \[\lambda(x):=\lim_{k\to\infty}\frac{\rho(4^{k}x)}{\sqrt{4^{k}x}}\] where $\rho(x):=\rho(\lfloor x\rfloor)$ for any $x> 0$. The function $\lambda$ is continuous and non-differentiable almost everywhere, for detail see \cite{CLWW}. Further, $\lambda(x)$ is self-similar in the sense that $\lambda(x)=\lambda(4x)$ for any $x>0$. The graph of $\lambda(x)$ on $[1,4]$, which is illustrated in figure \ref{fig:lambda}, has potential to be a fractal curve; and it is.  In fact, we prove the following result.

\begin{result2}
The box dimension of the graph of $\lambda(x)$ on any sub-interval of $(0,+\infty)$ is $3/2$.
\end{result2}

A variety of interesting fractals, both of theoretical and practical importance, occur as graphs of functions. Yue proved in \cite{Yue95} that the graph of one limit function studied in \cite{BEM} also has box dimension $3/2$. With a full probability, one dimensional Brownian sample function has Hausdorff dimension and box dimension $3/2$, see \cite[Theorem 16.4]{F04}. For any $b\geq 2$, the graph of Weierstarss function $W(x)=\sum_{n=0}^{\infty} b^{-n/2}\cos (b^{n}x)$ has Hausdorff dimension and box dimension $3/2$, see for example \cite{F04,Shen15} and references therein. For the Hausdorff dimension of the graph of $\lambda(x)$, Theorem B poses a good candidate $3/2$. It is natural to conjecture that the Hausdorff dimensions of the graphs of $\lambda(x)$ equals $3/2$.

The outline of this paper is as follows. In Section 2, we state basic definitions and notation. In Section 3, we give the recurrence relations
of the abelian complexity function of the Rudin-Shapiro sequence $\mathbf{r}$ and $\mathbf{r^{\prime}}$.  As a consequence, the
abelian complexity function of the Rudin-Shapiro sequence is $2$-regular, and the first difference of the abelian complexity function of the Rudin-Shapiro sequence is $2$-automatic. In the last section the box dimension of the graph of the function $\lambda(x)$ is studied.

\section{Preliminary}
In this section, we will introduce some notation and give the definitions of the abelian complexity function and the Rudin-Shapiro sequence.
\subsection{Finite and infinite words}
An \emph{alphabet} $\mathcal{A}$ is a finite and non-empty set (of symbols) whose elements are called \emph{letters}. A (finite)
\emph{word} over the  alphabet $\mathcal{A}$ is a concatenation of letters in $\mathcal{A}$. The concatenation of two words ${u} =
u(0)u(1) \cdots u(m)$ and ${v} = v(0)v(1) \cdots v(n)$ is the word ${uv} = u(0)u(1) \cdots u(m)v(0)v(1) \cdots v(n)$. The set of all finite
words over $\mathcal{A}$ including the \emph{empty word} $\varepsilon $ is denoted by $\mathcal{A}^*$. An infinite word $\mathbf{w}$ is an
infinite sequence of letters in $\mathcal{A}$. The set of all infinite words over $\mathcal{A}$ is denoted by $\mathcal{A}^{\mathbb{N}}$.

The \emph{length} of a finite word ${w}\in \mathcal{A^*}$, denoted by $|w|$, is the number of letters contained in $w$. We set $\left|
\varepsilon  \right| = 0$. For any word $u\in\mathcal{A}^{*}$ and any letter $a \in \mathcal{A}$, denote by $|{u}|_a$ the number of occurrences
of $a$ in ${u}$.

A word $w$ is a factor of a finite (or an infinite) word $v$, written by $w\prec v$ if there exist a finite word $x$ and a finite (or an
infinite) word $y$ such that $v=xwy$. When $x=\varepsilon$,  ${w}$ is called a \emph{prefix} of ${v}$, denoted by ${w} \triangleleft {v}$; when
$y=\varepsilon$, $ w$ is called a suffix of ${v}$, denoted by ${w} \triangleright {v}$.

\subsection{Digit sums}
Now we assume that the alphabet $\mathcal{A}$ is composed of integers. Let $\mathbf{w}=w(0)w(1)w(2)\cdots \in \mathcal{A}^{\mathbb{N}}$ be an
infinite word. For any $i\geq 0$ and $n\geq 1$, the sum of consecutive $n$ letters in $\mathbf{w}$ starting from the position $i$ is denoted by
\[\Sigma_{\mathbf{w}}(i,n):=\sum_{j=i}^{i+n-1} w(j).\]
The \emph{maximal sum} and \emph{minimal sum} of consecutive $n~(n\geq 1)$ letters in $\mathbf{w}$ are denoted by
\[ M_{\mathbf{w}}(n):=\max_{i\geq 0}\Sigma_{\mathbf{w}}(i,n) \textrm{ and } %
 m_{\mathbf{w}}(n):=\min_{i\geq 0}\Sigma_{\mathbf{w}}(i,n).\]
In addition, we always assume that $M_{\mathbf{w}}(0)=m_{\mathbf{w}}(0)=0$.

Denote the digit sum of a finite word ${u}=u(0)\cdots u(|{u}|-1)\in\mathcal{A}^{*}$ by
\[\mathrm{DS}(u):= \sum_{j=0}^{|{u}|-1} u(j),\]
then \[M_{\mathbf{w}}(n)=\max\big\{\mathrm{DS}(v):v\in\mathcal{F}_{\mathbf{w}}(n)\big\}\] and
\[m_{\mathbf{w}}(n)=\min\big\{\mathrm{DS}(v):v\in\mathcal{F}_{\mathbf{w}}(n)\big\}.\]

\medskip
The abelian complexity function of an infinite word $\mathbf{w}$ over $\{-1,1\}$ is closely related to the digit sums of factors of $\mathbf{w}$.
\begin{proposition}\label{lem:aMm}
Let $\mathbf{w}\in\{-1,1\}^{\mathbb{N}}$. Then \[\rho_{\mathbf{w}}(n)=\frac{M_{\mathbf{w}}(n)-m_{\mathbf{w}}(n)}{2}+1.\]
\end{proposition}
\begin{proof}For a proof one can refer to \cite[Proposition 2.2]{BBT11}.
\end{proof}

\subsection{The Rudin-Shapiro sequence $\mathbf{r}$ and a related sequence $\mathbf{r^{\prime}}$}
The Rudin-Shapiro sequence \[\mathbf{r}=r(0)r(1)\cdots r(n)\cdots \in {\{-1,1\}}^{\mathbb{N}}\] is given the following recurrence relations:
\begin{equation}\label{eq:recurrent}
 r(0)=1, ~r(2n)=r(n),~r(2n+1)=(-1)^nr(n) \quad (n\geq 0).
\end{equation}
The generating function $R(z)=\sum_{n\geq 0}r(n)z^{n}$ of the Rudin-Shapiro sequence satisfies the following Mahler type functional
equation
\[R(z)+R(-z)=2R(z^{2}).\]
We also study the coefficient sequence of $R(-z)$, denoted by \[\mathbf{r^{\prime}}=r^{\prime}(0)r^{\prime}(1)\cdots\in\{-1,1\}^{\mathbb{N}}.\]
Apparently, $r^{\prime}(n)=(-1)^{n}r(n)$ for all $n\geq 0$. Thus
\begin{equation}
r^{\prime}(0)=1,~r^{\prime}(2n)=(-1)^{n}r^{\prime}(n),~r^{\prime}(2n+1)=-r^{\prime}(n)\quad (n\geq 0).
\end{equation}

The Rudin-Shapiro sequence can also be generated by a substitution in the following way. Let  $\sigma:\{a,b,c,d\}\to\{a,b,c,d\}^*$ and $\tau,
\tau^{\prime}:\{a,b,c,d\}\to \{-1,1\}^{*}$ where
 \begin{equation*}
 \begin{array}{rllll}
 \sigma:& a\mapsto ab,& b\mapsto ac ,& c\mapsto db,& d\mapsto dc,\\
\tau: & a\mapsto 1, & b\mapsto 1, & c\mapsto -1,& d\mapsto -1,\\
\tau^{\prime}: & a\mapsto 1, & b\mapsto -1, & c\mapsto 1,& d\mapsto -1.
 \end{array}
 \end{equation*}
Let $\mathbf{s} := \sigma^{\infty}(a)$ be the fix point of $\sigma$ leading by $a$. Then  \[\mathbf{r} = \tau(\sigma^{\infty}(a)) \text{ and }
\mathbf{r^{\prime}}= \tau^{\prime}(\sigma^{\infty}(a)).\]

Denote by $\mathcal{M}_{\mathbf{s}}(n)$ (and $\mathcal{M}^{\prime}_{\mathbf{s}}(n)$) the set of all the factors of length $n$ in $\mathbf{s}$
such that the sum of letters of such factor under coding $\tau$ (and $\tau^{\prime}$, respectively) attains the maximal value, i.e.,
\begin{align*}
\mathcal{M}_{\mathbf{s}}(n) &:= \{ {u \in \mathcal{F}_{\mathbf{s}}(n)~:~ S(u) = M_{\mathbf{r}}(n)}\},\\
\mathcal{M}^{\prime}_{\mathbf{s}}(n) &:= \{ {u \in \mathcal{F}_{\mathbf{s}}(n)~:~ S^{\prime}(u) = M_{\mathbf{r^{\prime}}}(n)}\}
\end{align*}
where $S:= \text{DS}\circ\tau$  and $S^{\prime}:= \text{DS}\circ\tau^{\prime}$.

\section{The Regularity of the abelian Complexity of $\mathbf{r}$ and $\mathbf{r^{\prime}}$}
In this section, we will discuss the regularity of the abelian complexity function of the Rudin-Shapiro sequence $\mathbf{r}$ and the sequence
$\mathbf{r^{\prime}}$. From now on, unless otherwise stated, we always set $\mathcal{A} = \{-1,1\}.$

\subsection{Statement of results}
\begin{theorem} \label{thm:abelcomp}
For any $n\geq 1$, \[M_{\mathbf{r}}(n)=M_{\mathbf{r^{\prime}}}(n)=:M(n).\]
Moreover, $M(1)=1$, $M(2)=2$, $M(3)=3$ and for $n\geq 1,$
\begin{align*}
M(4n)&= 2M(n)+2, & M(4n+1)&=2M(n)+1,\\
M(4n+2)&=M(n)+M(n+1)+1,& M(4n+3)&=2M(n+1)+1.
\end{align*}
\end{theorem}
\begin{corollary}\label{cor:abelcomp}
The sequence $(M(n))_{\cc{n\geq 0}}$ is $2$-regular.
\end{corollary}
\begin{proof}
The result follows from Theorem \ref{thm:abelcomp} , \cite[Theorem 16.1.3 (e)]{ALL} and \cite[Theorem 2.9]{AS92}.
\end{proof}

For all $n\geq 0$, let \[\Delta M(n):= M(n+1)-M(n).\] The difference sequence $(\Delta M(n))_{n\geq 0}$ is characterized by the
following result.
\begin{corollary}\label{cor:abeldiff}
$\Delta M(i)=1$ for $0\leq i\leq 3$, and for $n \geq 1$,
\begin{equation}\label{eq:delta}
\left\{
\begin{array}{ccccl}
\Delta M(4n)  &=&-\Delta M(4n+3)&=&-1,\\
\Delta M(4n+1)&=&\Delta M(4n+2)&=&\Delta M(n).
\end{array}
\right.
\end{equation}
Moreover,  $(\Delta M(n))_{n\geq 0}$ is a $2$-automatic sequence.
\end{corollary}
\begin{proof}
The difference sequence $(\Delta M(n))_{n\geq 0}$ can be generated by the automaton given in Figure \ref{fig:coro1}.
\begin{figure}[htbp]
\centering
\begin{tikzpicture}[scale=0.8, every node/.style={scale=0.8}, state/.style={scale=0.8, circle solidus,draw,
inner sep=1pt,minimum size=12mm},>=stealth,->,auto,black]
\node[state] (a) {$q_{0}$ \nodepart{lower} $1$};
\node[state] (b) [right=15mm of a] {$q_{1}$ \nodepart{lower} $1$};
\node[state] (c) [right=15mm of b] {$q_{2}$ \nodepart{lower} $-1$};
\node (initial) [left=8mm of a] {Start};
\draw [->] (initial) to (a);
\draw [->] (b) to [out=120, in=60,loop,distance=10mm] node [above] {$1,2,3$} (b);
\draw [->] (a) to [out=120, in=60,loop,distance=10mm] node [above] {$0$} (a);
\draw [->] (a) to  node [above] {$1,2,3$} (b);
\draw [->] (c.200) to [bend left] node [below] {$3$} (b.-20);
\draw [->] (b.20) to [bend left] node [above] {$0$} (c.160);
\draw [->] (c) to [out=120, in=60,loop,distance=10mm] node [above] {$0,1,2$} (c);
\end{tikzpicture}
\caption{The automaton that generates $(\Delta M(n))_{n\geq 0}$.}\label{fig:coro1}
\end{figure}
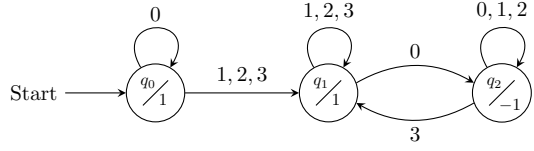
\end{proof}

\begin{theorem}\label{regular:abel}
For any $n\geq 1$,  \[\rho_{\mathbf{r}}(n)=\rho_{\mathbf{r^{\prime}}}(n):=\rho(n).\] Moreover, $(\rho(n))_{n\geq 0}$ is $2$-regular.
\end{theorem}
\begin{proof}
This result follows from Theorem \ref{thm:abelcomp} and Lemma \ref{lem:rho}.
\end{proof}

\subsection{Some lemmas}
To prove Theorem \ref{thm:abelcomp}, we need the following lemmas.
\begin{lemma}\label{lem:iterate}
For any word $w \in \{a,b,c,d\}^{*}$, we have
\[ S(\sigma^{2}(w))=2S(w) \text{ and } S^{\prime}(\sigma^{2}(w))=2S^{\prime}(w).  \]
\end{lemma}
\begin{proof}Observing that both $S$ and $S^{\prime}$ are morphism from $(\{a,b,c,d\}^{*},\cdot)$ to $(\mathbb{Z},+)$ where `$\cdot$' is the
concatenation of words, we only need to show the equalities in the lemma hold for any letter $x\in\{a,b,c,d\}$. By the definition of $\sigma$,
we get
$${\sigma^2}: a\mapsto abac, b\mapsto abdb , c\mapsto dcac, d\mapsto dcdb.$$
Recall that $\tau : a\mapsto 1, b\mapsto 1 , c\mapsto -1, d\mapsto -1$. Thus
\begin{align*}
S({\sigma ^2}(a)) &= S(abac) =\mathrm{DS}\circ\tau(abac)=\mathrm{DS}(111 (- 1)) = 2 = 2S(a).
\end{align*}
One can verify the rest cases in the same way.
\end{proof}
\begin{lemma}\label{lem:Mm}
For any $n\geq 1$, \[M_{\mathbf{a}}(n)+m_{\mathbf{a}}(n)=0,\]
where $\mathbf{a}$ represents the Rudin-Shapiro sequence $\mathbf{r}$ or the sequence $\mathbf{r^{\prime}}$.
\end{lemma}
\begin{proof}
We only prove the case $\mathbf{a}=\mathbf{r}$. The result for $\mathbf{a}=\mathbf{r^{\prime}}$ follows in the same way.

Let $\mu$ be the coding \[\mu: a\mapsto d, b\mapsto c, c\mapsto b, d\mapsto a.\] Then $\mu\circ\sigma=\sigma\circ\mu$ and $\mu\circ\mu={\rm
Id}$. We shall start by proving the following two facts: for any $W\in\{a,b,c,d\}^{n}$ ($n\geq 1$),
\begin{enumerate}
\item $W$ is a factor of $\mathbf{s}$ if and only if $\mu(W)$ is a factor of $\mathbf{s}$\ww{;}
\item $S(W)=M_{\mathbf{r}}(n)$ if and only if $S(\mu(W))=m_{\mathbf{r}}(n)$.
\end{enumerate}
For the fact $1$, if $W$ is a factor of $\mathbf{s}$, then $W$ is a factor of $\sigma^{k}(a)$ for some $k$. Therefore $\mu(W)$ is a factor of
$\mu(\sigma^{k}(a))=\sigma^{k}(d)$ which is a factor of $\sigma^{k+4}(a)$. Hence $\mu(W)$ is also a factor of $\mathbf{s}$. The converse holds
in the same argument by replacing $W$ by $\mu(W)$. Now we will prove fact $2$. Suppose $S(W)=M_{\mathbf{r}}(n)$ and $S(\mu(W))\neq
m_{\mathbf{r}}(n)$. Without lose of generality, assume that $S(\mu(W))> m_{\mathbf{r}}(n)$. This means there exists a word $W^{\prime}$ of
length $n$, such that $|W^{\prime}|_{-1}>|\mu(W)|_{-1}$. Therefore \[|\mu(W^{\prime})|_{1}=|W^{\prime}|_{-1}>|\mu(W)|_{-1}=|W|_{1}.\] It
follows that $M_{\mathbf{r}}(n)=S(W)<S(\mu(W^{\prime}))$ which is a contradiction. The converse can be proved by using the similar argument.

Noticing that $S(\mu(W))=-S(W)$, then by fact 1 and 2, the proof is completed.
\end{proof}

\begin{lemma}\label{lem:rho}
For any $n\geq 1$,
\[    \rho_{\mathbf{a}}(n) = M_{\mathbf{a}}(n)+1,  \]
where $\mathbf{a}$ represents the Rudin-Shapiro sequence $\mathbf{r}$ or the sequence $\mathbf{r^{\prime}}$.
\end{lemma}
\begin{proof}
The result follows from Proposition \ref{lem:aMm} and Lemma \ref{lem:Mm}.
\end{proof}

The following lemma characterizes digit sums $\Sigma_{\mathbf{r}}(\cdot,\cdot)$ which is useful in the study of $M_{\mathbf{r}}$.
\begin{lemma}\label{lem:sum}
For any $n\geq 1, i\geq 0$, we have
\begin{enumerate}[(1)]
\item $\Sigma_{\mathbf{r}}(4i,4n)  = 2\Sigma_{\mathbf{r}}(i,n),$
\item $  \Sigma_{\mathbf{r}}(4i + 1,4n)  = \Sigma_{\mathbf{r}}(i,n) + \Sigma_{\mathbf{r}}(i + 1,n), $
\item   $\Sigma_{\mathbf{r}}(4i + 2,4n)  = 2\Sigma_{\mathbf{r}}(i + 1,n),$
\item   $\Sigma_{\mathbf{r}}(4i + 3,4n) = 2\Sigma_{\mathbf{r}}(i + 1,n) - r(4i + 4n + 3) + r(4i + 3),  $
\item  $ \Sigma_{\mathbf{r}}(4i,4n + 1)  = 2\Sigma_{\mathbf{r}}(i,n) + r(i + n),$
\item  $ \Sigma_{\mathbf{r}}(4i + 1,4n + 1) = 2\Sigma_{\mathbf{r}}(i + 1,n) + r(i),  $
\item  $ \Sigma_{\mathbf{r}}(4i + 2,4n + 1) = 2\Sigma_{\mathbf{r}}(i + 1,n) + r(4i + 4n + 2),$
\item  $ \Sigma_{\mathbf{r}}(4i + 3,4n + 1) = 2\Sigma_{\mathbf{r}}(i + 1,n) + r(4i + 3); $
\item  $ \Sigma_{\mathbf{r}}(4i,4n + 2) = \Sigma_{\mathbf{r}}(i,n)+\Sigma_{\mathbf{r}}(i,n + 1)+r(i+n),$
\item  $ \Sigma_{\mathbf{r}}(4i + 1,4n + 2) = \Sigma_{\mathbf{r}}(i + 1,n) + \Sigma_{\mathbf{r}}(i,n + 1) + r(4i + 4n + 2),  $
\item  $ \Sigma_{\mathbf{r}}(4i + 2,4n + 2) = \Sigma_{\mathbf{r}}(i + 1,n)+\Sigma_{\mathbf{r}}(i + 1,n+1)-r(i+n+1),$
\item   $\Sigma_{\mathbf{r}}(4i + 3,4n + 2) = \Sigma_{\mathbf{r}}(i + 1,n) +\Sigma_{\mathbf{r}}(i + 1,n + 1) + r(4i + 3);$
\item  $ \Sigma_{\mathbf{r}}(4i,4n + 3) = 2\Sigma_{\mathbf{r}}(i,n + 1) - r(4i + 4n + 3),$
\item  $ \Sigma_{\mathbf{r}}(4i + 1,4n + 3) = 2\Sigma_{\mathbf{r}}(i,n + 1) - r(i), $
\item  $ \Sigma_{\mathbf{r}}(4i + 2,4n + 3) = 2\Sigma_{\mathbf{r}}(i + 1,n + 1) - r(i + n + 1),$
\item  $ \Sigma_{\mathbf{r}}(4i + 3,4n + 3) = 2\Sigma_{\mathbf{r}}(i + 1,n + 1) + r(4i + 3).$
\end{enumerate}
\end{lemma}
\begin{proof}
By \eqref{eq:recurrent} we have for all $n\geq 0$
$${r(4n)} = {r(4n + 1)} = r(n),~r(4n+2) =- r(4n+3) = {( - 1)^n}r(n). $$
Then by the previous equations and the definition of $\Sigma_{\mathbf{r}}$, these $16$ equations can be verified directly. Here we give the
proof of the first two equations as examples:
\begin{align*}
\Sigma_{\mathbf{r}}(4i,4n) & = \sum_{j = 4i}^{4i + 4n - 1} {r(j)}  \\
&=\sum_{j=i}^{i+n-1}(r(4j)+r(4j+1)+r(4j+2)+r(4j+3))\\
& =\sum_{j=i}^{i+n-1}(r(j)+r(j)+(-1)^{j}r(j)-(-1)^{j}r(j))\\
&=2\sum_{j=i}^{i+n-1}r(j)=2\Sigma_{\mathbf{r}}(i,n).\\
\Sigma_{\mathbf{r}}(4i+1,4n) & = \Sigma_{\mathbf{r}}(4i,4n)+r(4i+4n)-r(4i)=2\Sigma_{\mathbf{r}}(i,n)+r(i+n)-r(i)\\
&=\Sigma_{\mathbf{r}}(i,n)+\Sigma_{\mathbf{r}}(i+1,n);
\end{align*}
The rest equations can be proved in the same way.
\end{proof}
\begin{remark}
Lemma \ref{lem:sum} implies that the double sequence $(\Sigma_{\mathbf{r}})_{i\geq 0,n\geq 1}$ is a two-dimension {$2$}-regular sequence.
For a definition of two-dimensional regular sequences, see \cite{ALL}.
\end{remark}

The following lemma gives upper bounds of the maximal values of the sums of consecutive $n$ terms of $\mathbf{r}$ and
$\mathbf{r}^{\prime}$.
\begin{lemma}\label{lem:upperbound}
For any  $n\geq 1$,
\begin{eqnarray*}
M_{\mathbf{r}}(4n)& \leq & 2M_{\mathbf{r}}(n)+2,\\
M_{\mathbf{r}}(4n+1)&\leq & 2M_{\mathbf{r}}(n)+1,\\
M_{\mathbf{r}}(4n+2)&\leq & M_{\mathbf{r}}(n)+M_{\mathbf{r}}(n+1)+1,\\
M_{\mathbf{r}}(4n+3)&\leq & 2M_{\mathbf{r}}(n+1)+1.
\end{eqnarray*}
Moreover, the above inequalities also holds for $M_{\mathbf{r^{\prime}}}$.
\end{lemma}
\begin{proof}
For the first inequality, we shall use the first four equations of Lemma \ref{lem:sum}. By equations (1) to (3) of Lemma \ref{lem:sum}, we
obtain that for $k=0,1,2$,
\begin{align*}
\Sigma_{\mathbf{r}}(4i+k,4n) &\leq \max\{2\Sigma_{\mathbf{r}}(i,n),
\Sigma_{\mathbf{r}}(i,n)+\Sigma_{\mathbf{r}}(i+1,n),2\Sigma_{\mathbf{r}}(i+1,n)\}\\
&\leq 2M_{\mathbf{r}}(n).
\end{align*}
When $k=3$, by equation (4) of Lemma \ref{lem:sum}, we have
\begin{align*}
\Sigma_{\mathbf{r}}(4i+k,4n) &= 2\Sigma_{\mathbf{r}}(i+1,n)-r(4i+4n+3)+r(4i+3)\\
&\leq 2M_{\mathbf{r}}(n)+2.
\end{align*}
Therefore \(M_{\mathbf{r}}(4n)\leq 2M_{\mathbf{r}}(n)+2.\)

In a similar way, using the rest 12 equations of Lemma \ref{lem:sum}, we can prove
the rest three inequalities for $M_{\mathbf{r}}$.

 To prove the result for $M_{\mathbf{r^{\prime}}}$, one can deduce a similar result to Lemma \ref{lem:sum} for $\mathbf{r^{\prime}}$, and apply
 the similar argument as above. We left the details to the reader.
\end{proof}

\subsection{Proof of Theorem \ref{thm:abelcomp}}
To prove Theorem \ref{thm:abelcomp}, we only need to show that all equalities in Lemma \ref{lem:upperbound} hold. For this, we will construct
two sequences of words which attain the upper bounds in Lemma \ref{lem:upperbound} for $\mathbf{r}$ and $\mathbf{r^{\prime}}$
respectively. These will be done in the following Lemma \ref{lem:maxfactor} and \ref{lem:maxfactor2}. Then Theorem \ref{thm:abelcomp}
follows directly  from Lemma \ref{lem:upperbound}, \ref{lem:maxfactor} and \ref{lem:maxfactor2}.

Now we will give the sequence of words for $\mathbf{r}$. Let $(W_{n})_{n\geq 1}$ be the sequence of words defined by   $W_{1}=a$,
$W_{2}=ba$, $W_{3}=aba$ and
\begin{equation}\label{eq:maxW}\left\{
\begin{array}{ccl}
   {W_{4n}} &=& b{\sigma ^2}({W_n}){c^{ - 1}},  \\
   {W_{4n + 1}}& =& b{\sigma ^2}({W_n}),  \\
   {W_{4n + 2}} &= &
   \begin{cases}
     b{\sigma ^2}(W_{n+1}){(bac)^{ - 1}}& \text{ if } \Delta M_{\mathbf{r}}(n) = 1, \\
     cdb{\sigma ^2}({W_n}){c^{-1}} & \text{ if } \Delta M_{\mathbf{r}}(n) =  - 1,
      \end{cases}\\
   {W_{4n + 3}} &=& {\sigma ^2}({W_{n + 1}}){c^{ - 1}}.
\end{array}\right.
\end{equation}

\begin{lemma}\label{lem:maxfactor}
Let $(W_{n})_{n\geq 1}\subset\{a,b,c,d\}^{*}$ given by (\ref{eq:maxW}). Then for any $n\geq 1$,
\begin{enumerate}[\indent$(i)$]
\item  either $bW_{n}\prec \mathbf{s}$ or $dW_{n}\prec\mathbf{s}$ holds;\label{l:r2}
\item  either $a\triangleright W_{n}$ or $c\triangleright W_{n}$ holds; \label{l:r3}
\item $W_{n}\in\mathcal{M}_{\mathbf{s}}(n)$.\label{l:r1}
\end{enumerate}
\end{lemma}
\begin{proof}
We shall prove $(\ref{l:r2})$, $(\ref{l:r3})$ and $(\ref{l:r1})$ simultaneously by induction.

\emph{Step 1.} We shall show that the results hold for $n< 8$. Let $(W_{n})_{n=1}^{7}$ be the words given in table \ref{tab:1}. For
$n=1,2,3,4$, apparently $M_{\mathbf{r}}(n)=S(W_{n})$ which implies $W_{n}\in\mathcal{M}_{\mathbf{s}}(n)$. Since
$S(W_{5})=2M_{\mathbf{r}}(1)+1$, $S(W_{6})=M_{\mathbf{r}}(1)+M_{\mathbf{r}}(2)+1$ and $S(W_{7})=2M_{\mathbf{r}}(2)+1$, by Lemma
\ref{lem:upperbound}, we have $S(W_{n})=M_{\mathbf{r}}(n)$ and $W_{n}\in\mathcal{M}_{\mathbf{s}}(n)$ for $n=5,6,7$. Therefore $(\ref{l:r1})$
holds for $n< 8$. Notice that $(W_{n})_{n=1}^{7}$ are factors of $\sigma^{2}(dba)=dcdbabdbabac$ which is a factor of $\mathbf{s}$,
$(\ref{l:r2})$ and $(\ref{l:r3})$ also hold for $n< 8$.
\begin{table}[htbp]
\centering
\begin{tabular}{|>{$}c<{$}|>{$}c<{$}|>{$}c<{$}|>{$}c<{$}|>{$}c<{$}|>{$}c<{$}|>{$}c<{$}|>{$}c<{$}|}
\hline
n & 1 & 2 & 3 & 4 & 5 & 6 & 7\\ \hline
W_{n} & a & ba & aba & baba & babac & babdba & abdbaba \\ \hline
M_{\mathbf{r}}(n) & 1 & 2 & 3 & 4 & 3 & 4 & 5  \\ \hline
\end{tabular}
\caption{The initial values for Lemma \ref{lem:maxfactor}}\label{tab:1}
\end{table}

\emph{Step 2.} Assuming that $(\ref{l:r2})$, $(\ref{l:r3})$ and $(\ref{l:r1})$ hold for $n< 4k$ $(k\geq 2)$, we will prove the results for
$4k\leq n <4(k+1)$. The proof in this step will be separated into the following two cases.

\noindent\textbf{Case 1:} $\Delta M_{\mathbf{r}}(k)=1$. In this case,  the induction hypotheses $(\ref{l:r2})$, $(\ref{l:r3})$ and
$(\ref{l:r1})$ yield the  following facts:
\begin{enumerate}[\indent ({1}a)]
\item $W_{k}\in\mathcal{M}_{\mathbf{s}}(k)$ and $W_{k+1}\in\mathcal{M}_{\mathbf{s}}(k+1)$;
\item $db\sigma^{2}(W_{k})$ and $db\sigma^{2}(W_{k+1})$ are factors of $\mathbf{s}$;
\item either $a\triangleright W_{k}$ or $c\triangleright W_{k}$ holds, and $a\triangleright W_{k+1}$.
\end{enumerate}
(In the last statement (1c), we can exclude the case $c\triangleright W_{k+1}$ since $\Delta M_{\mathbf{r}}(k)=1$. In fact, if $W_{k+1}=Wc$,
then
\[M_{\mathbf{r}}(k+1)=S(W_{k+1})=S(W)+S(c)=S(W)-1\leq M_{\mathbf{r}}(k)-1,\]
which contradicts the assumption $\Delta M_{\mathbf{r}}(k)=M_{\mathbf{r}}(k+1)-M_{\mathbf{r}}(k)=1$.)
Now, by (\ref{eq:maxW}) and (1b), we have  $dW_{n}$ is a factor of $\mathbf{s}$ for $4k\leq n\leq 4k+2$ and $bW_{4k+3}$ is a factor of
$\mathbf{s},$ which implies that $(\ref{l:r2})$ holds for $4k\leq n<4(k+1)$. Moreover, this also implies
\begin{equation}\label{eq:2b}
W_{n} \text{ is a factor of }\mathbf{s} \text{ for }4k\leq n < 4(k+1).
\end{equation}
Since by the fact (1c), we have $ac\triangleright \sigma^{2}(W_{k})$ and $abac=\sigma^{2}(a)\triangleright\sigma^{2}(W_{k+1})$. Therefore
(\ref{eq:maxW}) gives
\begin{equation}\label{eq:2c}
a\triangleright W_{4k},~c\triangleright W_{4k+1},~ a\triangleright W_{4k+2} \text{ and } a\triangleright W_{4k+3},
\end{equation}
which prove $(\ref{l:r3})$.
Now, by (\ref{eq:maxW}), (\ref{eq:2c}), (1a) and Lemma \ref{lem:iterate}, we have
\begin{equation}\label{eq:2rec}
\left\{\begin{array}{ccl}
S(W_{4k}) & = & S(b)+ S(\sigma^{2}(W_{k}))-S(c) = 2M_{\mathbf{r}}(k)+2,\\
S(W_{4k+1}) & = & S(b)+ S(\sigma^{2}(W_{k})) = 2M_{\mathbf{r}}(k)+1,\\
S(W_{4k+2}) & = & S(b)+ S(\sigma^{2}(W_{k+1}))-S(bac)\\
& = & 2M_{\mathbf{r}}(k+1)=M_{\mathbf{r}}(k)+M_{\mathbf{r}}(k+1)+1,\\
S(W_{4k+3}) & = & S(\sigma^{2}(W_{k+1}))-S(c) = 2M_{\mathbf{r}}(k+1)+1.\\
\end{array}
\right.
\end{equation}
By (\ref{eq:2b}), (\ref{eq:2rec}) and Lemma \ref{lem:upperbound}, we have $W_{n}\in\mathcal{M}_{\mathbf{s}}(n)$ for $4k\leq n < 4(k+1)$ which
is $(\ref{l:r1})$.

\noindent\textbf{Case 2:} $\Delta M_{\mathbf{r}}(k)=-1$.  In this case, we shall first assert that $dW_{k}$ is a factor of $\mathbf{s}$. By the
induction hypothesis $(\ref{l:r2})$, we only need to show that $bW_{k}$ can not be a factor of $\mathbf{s}$. If this is not the case, then
\[M_{\mathbf{r}}(k+1)\geq S(bW_{k})=1+S(W_{k})=1+M_{\mathbf{r}}(k)\]
where the last equality follows from $(\ref{l:r1})$. Then we have $\Delta M_{\mathbf{r}}(k)=M_{\mathbf{r}}(k+1)-M_{\mathbf{r}}(k)\geq 1$ which
is a contradiction. Therefore, applying the induction hypotheses $(\ref{l:r2})$, $(\ref{l:r3})$ and $(\ref{l:r1})$, we have
\begin{enumerate}[\indent (2a)]
\item $W_{k}\in\mathcal{M}_{\mathbf{s}}(k)$ and $W_{k+1}\in\mathcal{M}_{\mathbf{s}}(k+1)$;
\item $dcdb\sigma^{2}(W_{k})$ and $b\sigma^{2}(W_{k+1})$ are factors of $\mathbf{s}$;
\item $ac\triangleright \sigma^{2}(W_{k})$ and $ac\triangleright \sigma^{2}(W_{k+1})$.
\end{enumerate}
By (\ref{eq:maxW}) and (2b), we have
\begin{equation}\label{eq:1bb}
dW_{n} \text{ is a factor of } \mathbf{s} \text{ for  } 4k\leq n\leq 4k+2
\end{equation}
and $bW_{4k+3}$ is a factor of $\mathbf{s}$, which prove $(\ref{l:r2})$. These imply that
\begin{equation}\label{eq:1b}
W_{n} \text{ is a factor of }\mathbf{s} \text{ for }4k\leq n < 4(k+1).
\end{equation}
Combing (2c) and (\ref{eq:maxW}), $(\ref{l:r3})$ holds for $4k\leq n < 4(k+1)$.
Now, by (\ref{eq:maxW}), (2a), (2c) and Lemma \ref{lem:iterate}, we have
\begin{equation}\label{eq:2rec2}
\left\{\begin{array}{ccl}
S(W_{4k}) & = & S(b)+ S(\sigma^{2}(W_{k}))-S(c) = 2M_{\mathbf{r}}(k)+2,\\
S(W_{4k+1}) & = & S(b)+ S(\sigma^{2}(W_{k})) = 2M_{\mathbf{r}}(k)+1,\\
S(W_{4k+2}) & = & S(cbd)+ S(\sigma^{2}(W_{k}))-S(c)\\
& = & 2M_{\mathbf{r}}(k)=M_{\mathbf{r}}(k)+M_{\mathbf{r}}(k+1)+1,\\
S(W_{4k+3}) & = & S(\sigma^{2}(W_{k+1}))-S(c) = 2M_{\mathbf{r}}(k+1)+1.\\
\end{array}
\right.
\end{equation}
By (\ref{eq:1b}), (\ref{eq:2rec2}) and Lemma \ref{lem:upperbound}, we have $W_{n}\in\mathcal{M}_{\mathbf{s}}(n)$ for $4k\leq n < 4(k+1)$ which
is $(\ref{l:r1})$. The proof is completed.
\end{proof}

\medskip
For $\mathbf{r^{\prime}}$, let $(\widetilde{W}_{n})_{n\geq 1}$ be the sequence of words defined by   $\widetilde{W}_{1}=c$,
$\widetilde{W}_{2}=ca$, $\widetilde{W}_{3}=cac$ and
\begin{equation}\label{eq:maxW2}\left\{
\begin{array}{ccl}
   {\widetilde{W}_{4n}} &=& d^{-1}{\sigma ^2}({\widetilde{W}_n}){a},  \\
   {\widetilde{W}_{4n + 1}}& =& {\sigma ^2}({\widetilde{W}_n})a,  \\
   {\widetilde{W}_{4n + 2}} &= &
   \begin{cases}
      (dca)^{-1}{\sigma ^2}(\widetilde{W}_{n+1}){a}& \text{ if } \Delta M_{\mathbf{r^{\prime}}}(n) = 1, \\
      d^{-1}{\sigma ^2}({\widetilde{W}_n}){abd} & \text{ if } \Delta M_{\mathbf{r^{\prime}}}(n) =  - 1,
      \end{cases}\\
   {\widetilde{W}_{4n + 3}} &=& d^{-1}{\sigma ^2}({\widetilde{W}_{n + 1}}).
\end{array}\right.
\end{equation}

\begin{lemma}\label{lem:maxfactor2}
Let $(\widetilde{W}_{n})_{n\geq 1}\subset\{a,b,c,d\}^{*}$ given by (\ref{eq:maxW2}). Then for any $n\geq 1$,
\begin{enumerate}[\indent$(i)$]
\item either $\widetilde{W}_{n}a\prec\mathbf{s}$ or $\widetilde{W}_{n}b\prec\mathbf{s}$ holds;\label{l:r2new}
\item either $c\triangleleft \widetilde{W}_{n}$ or $d\triangleleft \widetilde{W}_{n}$ holds; \label{l:r3new}
\item $\widetilde{W}_{n}\in\mathcal{M}'_{\mathbf{s}}(n)$.\label{l:r1new}
\end{enumerate}
\end{lemma}
\begin{proof}
The proof of this lemma is similar to Lemma \ref{lem:maxfactor}.
\end{proof}

For any  $k$-automatic sequence $\mathbf{w}=w(0)w(1)\cdots\in\{-1,1\}^{\mathbb{N}},$ 
the regularity of the maximal partial sums $(M_{\mathbf{w}}(n))_{n\geq 1}$ and the minimal partial sums $(m_{\mathbf{w}}(n))_{n\geq 1}$ implies the
regularity of the abelian complexity $(\rho_{\mathbf{w}}(n))_{n\geq 1}.$
By proving the same result as Lemma \ref{lem:sum}, one can show that the double sequence $(\Sigma_{\mathbf{w}}(i,n))_{i\geq 0,n\geq
1}$ is 2-dimensional $k$-regular.  In fact,  it is not hard to show that  $(\Sigma_{\mathbf{w}}(i,n))_{i\geq 0}$ is $k$-automatic for any
fixed $n\geq 1$, and  $(\Sigma_{\mathbf{w}}(i,n))_{n\geq 1}$ is $k$-regular for any fixed $i\geq 0$. Moreover, Theorem \ref{thm:abelcomp}
and Lemma \ref{lem:Mm} show that $(\max_{i\geq 0}\Sigma_{\bm{w}}(i,n))_{n\geq 1}$  and $(\min_{i\geq 0}\Sigma_{\mathbf{w}}(i,n))_{n\geq
1}$ are still $k$-regular when $\mathbf{w}$ is the Rudin-Shapiro sequence $\mathbf{r}$ or its related sequence
$\mathbf{r}^{\prime}$, which implies the regularity of the abelian complexity function $(\rho_{\mathbf{r}}(n))_{n\geq 0}$ and $(\rho_{\mathbf{r}^{\prime}}(n))_{n\geq 0}.$ It is natural to
ask  whether $(\max_{i\geq 0}\Sigma_{\mathbf{w}}(i,n))_{n\geq 1}$  and $(\min_{i\geq 0}\Sigma_{\mathbf{w}}(i,n))_{n\geq 1}$ are always
$k$-regular for general $k$-automatic sequences $\mathbf{w}$ over $\{-1,1\}.$

\section{Box dimension of $\lambda(x)$}	
Let $M(x):=M(\lfloor x\rfloor)$ ($x>0$) be the continuous version of the maximal digit sum function, and $\rho(x)=M(x)+1$. Now we study the following limit function:
\begin{equation}\label{def:lamda}
\lambda(x):= \displaystyle{\lim_{k\to \infty}}\frac{\rho( 4^k x)}{\sqrt{4^k x}}.
\end{equation}
From the above definition, providing the limit exists, it is easy to see that $\lambda(x)$ is self-similar in the sense that for any $x>0$,
 	\[  \lambda(4x)=\lambda(x).\]
The existence of the limit in \eqref{def:lamda} follows from the same argument in \cite[Theorem $1$]{BEM}. For completeness, we give the details in the
following Proposition \ref{lem:lambdarho}.

Denote the $4$-adic expansion of a real positive number $x>0$ by
\begin{equation}\label{eq:expansion}
\sum_{j=0}^{\infty}x_{j}4^{-j}
\end{equation}
where $x_{0}\in\mathbb{N}$ and $x_{j}\in\{0,1,2,3\}$ for all $j\geq 1$.  In the expansion \eqref{eq:expansion}, we always assume that there are
infinitely many $j$ such that $x_{j}\neq 3$. Let
\[a_{j}(x):=\left\{
\begin{aligned}
& -1, & \text{if } 4^{j}x<1,\\
&\Delta M(\lfloor 4^{j}x\rfloor-1), & \text{otherwise,}
\end{aligned}
\right.\]
and
\[d(y)=\left\{\begin{array}{cl} 1 & \text{ if } y = 0 \text{ or } 2,\\
0 & \text{ if } y = 1,\\
2 & \text{ if } y = 3.\end{array}\right.\]
\begin{proposition}\label{lem:lambdarho}
The limit (\ref{def:lamda}) exists for all $x>0$, and for any $x>0$ it satisfies
\begin{equation}\label{eq:lamda:detail}
\lambda(x)=\frac{\rho(x)+a(x)}{\sqrt{x}}
\end{equation}
where $a(x) := \sum\limits_{j = 1}^\infty  d ({x_j}){a_j}(x) 2^{ - j}.$ Moreover, for any positive integer $n$, \[\lambda(n)=(\rho(n)+1)/\sqrt{n}.\]
\end{proposition}
\begin{proof}
By Theorem \ref{thm:abelcomp} and Corollary \ref{cor:abeldiff}, we have \[M(4n+i)=2M(n)+1+d(i)\Delta M(4n+i-1)\] for all $n\geq 1$
and $i=0,1,2,3$.
Let $N$ be the smallest integer such that $4^{N}x\geq 1$. Then, for any $k\geq N$,
\begin{align*}
M(4^{k}x)&=M(\lfloor 4^{k}x\rfloor) = M(4\lfloor 4^{k-1}x\rfloor +x_{k})\\
& = 2M(\lfloor 4^{k-1}x\rfloor)+1+d(x_{k})\Delta M(\lfloor 4^{k}x\rfloor -1)\\
& = 2M(4^{k-1}x)+1+d(x_{k})a_{k}(x).
\end{align*}
For $1\leq k < N$, $d(x_{k})=d(0)=1$ and $a_{k}(x)=-1$. Thus, we also have
\begin{align*}
M(4^{k}x)&=0=1+(-1) \\
& = 1+ d(x_{k})a_{k}(x)\\
& = 2M(4^{k-1}x)+1+d(x_{k})a_{k}(x).
\end{align*}
By induction, the above equation yields \[M(4^{k}x)=2^{k}M(x)+\sum_{j=1}^{k}d(x_{j})a_{j}(x)2^{k-j}+(2^{k}-1).\]
Now, by Lemma \ref{lem:rho}
\[\frac{\rho ({4^k}x)}{\sqrt {{4^k}x} } = \frac{M({4^k}x) + 1}{\sqrt {{4^k}x} } = \frac{\rho (x)}{\sqrt x } + \frac{1}{\sqrt x }\sum\limits_{j
= 1}^k d({x_j})a_j(x)2^{ - j}.\]
Letting $k\to \infty$ and noticing that the series in (\ref{eq:lamda:detail}) converges absolutely, we obtain (\ref{eq:lamda:detail}).

When $x=n\in\mathbb{N^{+}}$, $x_{0}=n$, $x_j=0$ and $a_{j}=4^{j}n-1$ for all $j\geq 1$. Then the infinite sums in (\ref{eq:lamda:detail}) turns
out to be
\[\sum_{j=1}^{\infty}d(x_{j})a_{j}(x)2^{-j}=\sum_{j=1}^{\infty}\Delta M(4^{j}n-1)2^{-j}=1\]
where the last equality holds by using Corollary \ref{cor:abeldiff}. Applying the above equation to (\ref{eq:lamda:detail}), we complete
the proof.
\end{proof}

\subsection{Auxiliary lemmas.}
Let $\delta > 0$. For any $m_1,m_2 \in \mathbb{Z}$,  we call the following square
\[ [m_1\delta, (m_1+1)\delta] \times  [m_2\delta, (m_2+1)\delta]\]
a $\delta$-mesh of $\mathbb{R}^2.$
Let $F \subset \mathbb{R}^2$  be a non-empty bounded set in $\mathbb{R}^{2}$, and $N_{\delta}(F)$ be the number of $\delta$-meshes that intersect
$F$. The upper and lower box dimension are defined by
\[\overline{\dim}_B F:=\overline{\lim}_{\delta\to 0}\frac{\log N_{\delta}(F)}{-\log
\delta} \text{  and  } \underline{\dim}_B F:=\overline{\lim}_{\delta\to 0}\frac{\log N_{\delta}(F)}{-\log \delta}\] respectively. If
$\overline{\dim}_B F= \underline{\dim}_B F,$ then the common value denoted by $\dim_B F$,  is the box dimension of $F$. For more detail, see \cite{F04}.

Now, we will prove some auxiliary lemmas which are used in the calculation of the box dimension of the function $\lambda(x)$. For any $k\geq 1$ and $0\leq z< 4^{k}$ where $z\in\mathbb{N}$. let
\[I_{k}(z):=[z4^{-k},(z+1)4^{-k}).\]
Then $[0,1)=\bigcup_{0\leq z<4^{k}}I_{k}(z)$. Denote the $4$-adic expansion of $z4^{-k}$ by  \[\frac{z}{4^{k}}=\sum_{j=1}^{k}z_{j}4^{-j}.\]
If $y=\sum_{j=1}^{\infty}y_{j}4^{-j}\in I_{k}(z)$, then $y_{i}=z_{i}$ for $i=1,~2,~\cdots,~k$.

First, we will determine the difference of values of $a(\cdot)$ at the end points of $4$-adic interval $I_{k}(z)$.
\begin{lemma}\label{dim:lem:1}
Let $k\geq 1$ and $z\in\mathbb{N}$ with $1\leq z<4^{k}$. Then
\[a(z4^{-k})-a((z+1)4^{-k})=\begin{cases}-2^{-k} & \mathrm{if}~ z\leq 4^{k}-2 \\
1-2^{-k} & \mathrm{if }~ z=4^{k}-1.\end{cases}\]
\end{lemma}
\begin{proof}
When $z=4^{k}-1$, we have $z4^{-k}=\sum_{j=1}^{k}3\cdot 4^{-j}$ and $(z+1)4^{-k}=1$. So \[a(z4^{-k})-a((z+1)4^{-k})=(2-2^{-k})-1=1-2^{-k}.\]

When $1\leq z\leq 4^{-k}-2$, $z4^{-k}$ and $(z+1)4^{-k}$ have the $4$-adic expansions
\[z4^{-k}=\sum_{j=1}^{k}z_{j}4^{-j} \text{ and }(z+1)4^{-k}=\sum_{j=1}^{k}z^{\prime}_{j}4^{-j}.\] Implicitly, we assume that
$z_{j}=z_{j}^{\prime}=0$ for $j>k$. Let $1\leq h\leq k$ be the integer such that $z_{h}\neq 3$ and $z_{j}=3$ for $j=h+1,\cdots, k$. Then
\[z^{\prime}_{j}=\begin{cases}z_{j} & \text{when } j<h,\\
z_{j}+1 & \text{when } j = h,\\
0 & \text{when } j> h.\end{cases}\]
Setting $D_{j}:=d(z_{j})a_{j}(z4^{-k})-d(z^{\prime}_{j})a_{j}((z+1)4^{-k})$, then \[a(z4^{-k})-a((z+1)4^{-k})=\sum_{j=1}^{\infty}D(j)2^{-j}.\]
Apparently, $D_{j}=0$ when $j<h$ or $j>k$. Since $a_{j}(z4^{-k})=a_{j}((z+1)4^{-k})=1$ for $h+2\leq j\leq k$, we have for $ h+2\leq j\leq k$,
\[D_{j}=d(3)-d(0)=1.\]
Set $u:=4^{h}\sum_{j=1}^{h}z_{j}4^{-j}$. If $u\geq 1$, we have
\begin{align*}
D_{h}+2^{-1}D_{h+1}&= \left(d(z_{h})\Delta M(u-1)-d(z_{h}^{\prime})\Delta M(u)\right)\\
& \quad +2^{-1}\left(d(3)\Delta M(4u+2)-d(0)\cdot \Delta M(4u+3)\right)\\
&= d(z_{h})\Delta M(u-1)-d(z_{h}^{\prime})\Delta M(u)+\Delta M(u)-2^{-1}\\
&= \begin{cases}
d(0)\cdot 1-d(1)\cdot (-1)+(-1)-2^{-1}, & \text{if } z_{h}=0,\\
d(1)\cdot (-1)-d(2)\cdot \Delta M(u)+\Delta M(u)-2^{-1}, & \text{if } z_{h}=1,\\
d(2)\cdot \Delta M(u)-d(3)\cdot \Delta M(u)+\Delta M(u)-2^{-1}, & \text{if } z_{h}=2,
\end{cases}\\
&= -2^{-1}.
\end{align*}
If $u=0$, then $z_{h}=0$ and
\begin{align*}
D_{h}+2^{-1}D_{h+1}&= d(0)\cdot(-1)-d(1)\Delta M(0)\\
&\quad + 2^{-1}\left(d(3)\Delta M(2)-d(0)\Delta M(3)\right)\\
& =  -2^{-1}.
\end{align*}
Therefore
\begin{align*}
a(z4^{-k})-a((z+1)4^{-k}) & = \sum_{j=1}^{\infty}D(j)2^{-j}\\
& = 2^{-h}D_{h}+2^{-h-1}D_{h+1}+ \sum_{j=h+2}^{k}D(j)2^{-j}\\
& = -2^{-h-1}+\left(2^{-h-1}-2^{-k}\right) =-2^{-k}.
\end{align*}
\end{proof}

\begin{lemma}\label{dim:lem:2}
There exists $c>0$, such that for any $x,y\in (0,1)$,
\[|a(x)-a(y)|\leq c|x-y|^{1/2}.\]
\end{lemma}
\begin{proof}
Let $x,y\in (0,1)$ and $x<y$. Denote their $4$-adic expansion by
\[x=\sum_{j=1}^{\infty}x_{j}4^{-j} \text{ and } y=\sum_{j=1}^{\infty}y_{j}4^{-j}.\]
Set $D_{j}:=d(x_{j})a_{j}(x)-d(y_{j})a_{j}(y)$, then $|D_{j}|\leq 4$ for $j\geq 1$.

Let $k$ be the integer such that $4^{-k-1}\leq y-x<4^{-k}$. Then $x$ and $y$ can be covered by at most two (adjacent) $4$-adic intervals of
level $k$. Suppose $x,y\in I_{k}(z)$ for some $0\leq z<4^{k}$, then $x_{j}=y_{j}$ for $i=1,2,\cdots,k$. Consequently, $D_{j}=0$ for $1\leq
j\leq k$. So
\begin{align*}
|a(x)-a(y)| & = \left|\sum_{j=k+1}^{\infty}D_{j}2^{-j}\right|\\
& \leq 4\sum_{j=k+1}^{\infty}2^{-j}=4\cdot 2^{-k}\\
& \leq 8|x-y|^{1/2}.
\end{align*}

On the other hand, suppose $x\in I_{k}(z)$ and $y\in I_{k}(z+1)$ where $0\leq z<4^{-k}-1$. Let $h$ be the largest integer such that $x,y\in
I_{h}(z^{\prime})$ for some $0\leq z^{\prime} <4^{-h}$. Apparently, $0\leq h<k$. In this case, the $4$-adic expansions of $x$ and $y$ satisfy
\begin{align*}
\begin{cases}y_{j}=x_{j}, & \mathrm{if}~ 1\leq j \leq h,\\
y_{j}=x_{j}+1, & \mathrm{if}~ j= h+1,\\
y_{j}=0~ \mathrm{and}~x_{j}=3, & \mathrm{if}~ h+2\leq j\leq k.\end{cases}
\end{align*}
(We remark that $x_{h+1}\neq 3$ by the choice of $h$.) Hence, $D_{j}=0$ for $1\leq j\leq h$. Similar discussions as in Lemma \ref{dim:lem:1}
yield that
\begin{align*}
D_{h+1}+2^{-1}D_{h+2}=-2^{-1}.
\end{align*}
Moreover, for $h+2\leq j\leq k$, $D_{j}=d(3)-d(0)=1$. Therefore,
\begin{align*}
|a(x)-a(y)|&=\left|\sum_{j=1}^{\infty}D_{j}2^{-j}\right|= \left|\sum_{j=h+1}^{k}D_{j}2^{-j}+\sum_{j=k+1}^{\infty}D_{j}2^{-j}\right|\\
& \leq \left|2^{-h-1}(D_{h+1}+2^{-1}D_{h+2})+\sum_{j=h+3}^{k}D_{j}2^{-j}\right|+4\sum_{j=k+1}^{\infty}2^{-j}\\
& = 5\cdot 2^{-k} \leq 10 |x-y|^{1/2}.
\end{align*}
\end{proof}

\subsection{Calculation of the box dimension.}

\begin{theorem} \label{dim:thm:1}
For any $0<\alpha<\beta\leq 1$,
\[\dim_{B}\{(x,\lambda(x)):\alpha<x<\beta\}=\frac{3}{2}.\]
\end{theorem}
\begin{proof}
For any $x,y\in (\alpha,\beta)$ and $x<y$, $\rho(x)=\rho(y)=\rho(0)=1$,
\begin{align*}
|\lambda(x)-\lambda(y)| & = \left|\frac{\rho(x)+a(x)}{\sqrt{x}}-\frac{\rho(y)+a(y)}{\sqrt{y}}\right|\\
&=\left|\frac{a(x)+1}{\sqrt{x}}-\frac{a(y)+1}{\sqrt{y}}\right|\\
&= \left|\frac{a(x)-a(y)}{\sqrt{x}}+\frac{\sqrt{y}-\sqrt{x}}{\sqrt{xy}}(a(y)+1)\right|\\
& \leq \alpha^{-1/2}|a(x)-a(y)|+3\alpha^{-1}\sqrt{y-x} \\
& \leq (c\alpha^{-1/2}+3\alpha^{-1})|x-y|^{1/2}
\end{align*}
where the last inequality holds by Lemma \ref{dim:lem:2}. Now by \cite[Corollary 11.2 (a)]{F04},
\begin{equation}
\overline{\dim}_{B}\{(x,\lambda(x)):\alpha<x<\beta\}\leq \frac{3}{2}.\label{dim:eq:u}
\end{equation}

For any $k\geq 1$, let $N_{k}$ be the number of $4^{-k}$-mesh squares that intersect the graph of $\lambda(x)$ on $(\alpha,\beta)$. For any
$k\geq 1$ and $\lfloor \alpha 4^{k}\rfloor< z \leq \lfloor \beta 4^{k}\rfloor$, the number of $4^{-k}$-mesh squares that intersect the graph of
$\lambda(x)$ on $I_{k}(z)$ is lager than $\left|\lambda((z+1)4^{-k})-\lambda(z4^{-k})\right|/4^{-k}$.
Choose $K_{1}$ large enough such that for all $k>K_{1}$, $3\cdot 2^k<\lfloor \alpha 4^k\rfloor ~(<z)$. Then, by Lemma \ref{dim:lem:1},
\begin{align*}
\left|\lambda((z+1)4^{-k})-\lambda(z4^{-k})\right| & =
\left|\frac{1+a((z+1)4^{-k})}{\sqrt{(z+1)4^{-k}}}-\frac{1+a(z4^{-k})}{\sqrt{z4^{-k}}}\right|\\
& = \frac{1}{\sqrt{z4^{-k}}}\left|a((z+1)4^{-k})-a(z4^{-k})\right.\\
& \quad +\left. \frac{\sqrt{z4^{-k}}-\sqrt{(z+1)4^{-k}}}{\sqrt{(z+1)4^{-k}}}(1+a((z+1)4^{-k}))\right|\\
& \geq \frac{1}{\sqrt{\beta}}\left(2^{-k}-\frac{\left|1+a((z+1)4^{-k})\right|}{z+1+\sqrt{z^{2}+z}}\right)\\
& \geq 2^{-k}\cdot \frac{1}{\sqrt{\beta}}\left(1-\frac{3\cdot 2^{k}}{z+1+\sqrt{z^{2}+z}}\right) >\frac{1}{2\sqrt{\beta}}\cdot 2^{-k}.
\end{align*}
Choose $K_{2}$ large enough such that for all $k>K_{2}$, $\lfloor \beta 4^{k}\rfloor-\lfloor \alpha 4^{k}\rfloor-1>4^{k}(\beta-\alpha)/2$.
Hence, for any $k>\max\{K_{1}, K_{2}\}$,
\begin{align*}
N_{k} & \geq \sum_{\lfloor \alpha 4^{k}\rfloor< z < \lfloor \beta
4^{k}\rfloor}\frac{\left|\lambda((z+1)4^{-k})-\lambda(z4^{-k})\right|}{4^{-k}}\\
& \geq \frac{1}{2\sqrt{\beta}}\sum_{\lfloor \alpha 4^{k}\rfloor< z < \lfloor \beta 4^{k}\rfloor}\frac{2^{-k}}{4^{-k}} =  \frac{\lfloor \beta
4^{k}\rfloor-\lfloor \alpha 4^{k}\rfloor-1}{2\sqrt{\beta}}\cdot 2^{k}\\
& > \frac{\beta-\alpha}{4\sqrt{\beta}}\cdot 2^{3k}.
\end{align*}
Therefore
\begin{align}
\underline{\dim}_{B}\{(x,\lambda(x)):\alpha<x<\beta\} & = \liminf_{k\to\infty}\frac{\log N_{k}}{-\log 4^{-k}}\notag \\
& \geq \liminf_{k\to\infty}\frac{\log \left(2^{3k}(\beta-\alpha)/{4\sqrt{\beta}}\right)}{-\log 4^{-k}}=\frac{3}{2}. \label{dim:eq:l}
\end{align}
The result follows from \eqref{dim:eq:u} and \eqref{dim:eq:l}.
\end{proof}

\begin{corollary}
For any $0<\alpha <\beta$,
\[\dim_{B}\{(x,\lambda(x)):\alpha<x<\beta\}=\frac{3}{2}.\]
\end{corollary}
\begin{proof}
Let $K$ be an integer such that $\beta/4^{K}\leq 1$.
Since $\lambda(4x)=\lambda(x)$ for $x>0$, the following mapping
\[f:(x,\lambda(x))\mapsto	(4^{K}x,\lambda(4^{K}x))\]
is a bi-Lipschitz mapping in $\mathbb{R}^{2}$, and
\begin{align*}
f\left(\{(x,\lambda(x)):4^{-K}\alpha<x<4^{-K}\beta\}\right) & = \left\{\left(4^{K}x,\lambda(4^{K}x)\right):4^{-K}\alpha<x<4^{-K}\beta\right\}\\
& = \{(y,\lambda(y)):\alpha<y<\beta\}.
\end{align*}
The result follows from Theorem \ref{dim:thm:1} and the above equation.
\end{proof}

\end{document}